\documentclass{amsart}

\newtheorem{theorem}{Theorem}[section]
\newtheorem{lemma}[theorem]{Lemma}
\newtheorem{corollary}[theorem]{Corollary}

\theoremstyle{definition}

\theoremstyle{remark}
\newtheorem{remark}[theorem]{Remark}

\usepackage{amsmath}
\usepackage{amssymb}
\usepackage{graphicx}
\usepackage{mathrsfs}
\usepackage{algorithm}
\usepackage{algorithmicx}
\usepackage{appendix}
\usepackage{enumitem}
\usepackage{stmaryrd}
\usepackage[colorlinks, linkcolor=blue, anchorcolor=blue, citecolor=blue]{hyperref}
\AtBeginDocument{
                 \label{CorrectFirstPageLabel}
                 
                }
\DeclareMathOperator*{\Null}{null}
\DeclareMathOperator*{\Range}{range}
\DeclareMathOperator*{\rank}{rank}

\DeclareMathOperator*{\tr}{tr}

\let\s=\scriptscriptstyle

\numberwithin{equation}{section}

\begin{document}

\title[Two-level methods with general coarse solvers]{Convergence analysis of two-level methods with general coarse solvers}

\author{Xuefeng Xu}
\address{Department of Mathematics, Purdue University, West Lafayette, IN 47907, USA}
\email{xuxuefeng@lsec.cc.ac.cn; xu1412@purdue.edu}






\begin{abstract}
Multilevel methods are among the most efficient numerical methods for solving large-scale linear systems that arise from discretized partial differential equations. The fundamental module of such methods is a two-level procedure, which consists of compatible relaxation and coarse-level correction. Regarding two-level convergence theory, most previous works focus on the case of exact (Galerkin) coarse solver. In practice, however, it is often too costly to solve the Galerkin coarse-level system exactly when its size is relatively large. Compared with the exact case, the convergence theory of inexact two-level methods is of more practical significance, while it is still less developed in the literature, especially when nonlinear coarse solvers are used. In this paper, we establish a general framework for analyzing the convergence of inexact two-level methods, in which the coarse-level system is solved approximately by an inner iterative procedure. The framework allows us to use various (linear, nonlinear, deterministic, randomized, or hybrid) solvers in the inner iterations, as long as the corresponding accuracy estimates are available.
\end{abstract}

\maketitle



\section{Introduction}

Multilevel methods are among the most powerful numerical techniques for solving large-scale linear systems that arise from discretized partial differential equations. These types of methods can be obtained by a change of basis with multilevel structure or a hierarchy of grids, such as the hierarchical basis~\cite{Yserentant1986,Bank1988} and multigrid methods~\cite{Hackbusch1985,Briggs2000,Trottenberg2001,XZ2017}. Most of multilevel methods share some common features, e.g., they are based on a recursive call of two-level procedure, which is a combination of two complementary error-reduction processes: \textit{compatible relaxation} and \textit{coarse-level correction}. Compatible relaxation, originated with Brandt~\cite{Brandt2000}, is a modified scheme that keeps the coarse variables invariant, which is a useful tool for selecting coarse-grids in algebraic multigrid methods~\cite{Falgout2004}. The purpose of coarse-level correction is to further reduce error components that cannot be effectively eliminated by the compatible relaxation. The optimality of multilevel methods will be achieved when these two processes complement each other very well.

Consider solving the linear system
\begin{equation}\label{system}
A\mathbf{u}=\mathbf{f},
\end{equation}
where $A\in\mathbb{R}^{n\times n}$ is symmetric positive definite (SPD), $\mathbf{u}\in\mathbb{R}^{n}$, and $\mathbf{f}\in\mathbb{R}^{n}$. Given an initial guess $\mathbf{u}^{(0)}\in\mathbb{R}^{n}$, we perform the following compatible relaxation process:
\begin{equation}\label{CR}
\mathbf{u}^{(k+1)}=\mathbf{u}^{(k)}+SM_{\rm s}^{-1}S^{T}\big(\mathbf{f}-A\mathbf{u}^{(k)}\big) \quad k=0,1,\ldots,
\end{equation}
where $S\in\mathbb{R}^{n\times n_{\rm s}}$ ($n_{\rm s}<n$) is of full column rank and $M_{\rm s}\in\mathbb{R}^{n_{\rm s}\times n_{\rm s}}$ is a nonsingular matrix such that $M_{\rm s}+M_{\rm s}^{T}-A_{\rm s}$ is SPD ($A_{\rm s}:=S^{T}AS$). The relaxation method~\eqref{CR} is essentially a smoothing process, in which $M_{\rm s}$ is treated as a \textit{local} smoother. We stress that $M_{\rm s}$ is not restricted to simple types (e.g., the Jacobi and Gauss--Seidel methods), because its size is relatively small. In particular, if $M_{\rm s}=A_{\rm s}$, we deduce from~\eqref{CR} that
\begin{displaymath}
\mathbf{u}-\mathbf{u}^{(k+1)}=\big(I-SA_{\rm s}^{-1}S^{T}A\big)\big(\mathbf{u}-\mathbf{u}^{(k)}\big).
\end{displaymath}
Note that $I-SA_{\rm s}^{-1}S^{T}A$ is an $A$-orthogonal projection along (or parallel to) $\Range(S)$ onto $\Null(S^{T}A)$. If $\Range(S)$ contains most of high-frequency (i.e., oscillatory) error, then these components will be eliminated effectively. The focus of the smoothing process~\eqref{CR} is on high-frequency error components. The remaining low-frequency (i.e., smooth) error will be further reduced by the coarse-level correction process. Let $\mathbf{u}^{(\ell)}\in\mathbb{R}^{n}$ be an approximation to $\mathbf{u}\equiv A^{-1}\mathbf{f}$ (e.g., $\mathbf{u}^{(\ell)}$ is generated from~\eqref{CR}), and let $P\in\mathbb{R}^{n\times n_{\rm c}}$ ($n_{\rm c}<n$) be a prolongation (or interpolation) matrix with full column rank. Then, the (exact) coarse-level correction can be described as follows:
\begin{equation}\label{correction}
\mathbf{u}^{(\ell+1)}=\mathbf{u}^{(\ell)}+PA_{\rm c}^{-1}P^{T}\big(\mathbf{f}-A\mathbf{u}^{(\ell)}\big),
\end{equation}
where $A_{\rm c}:=P^{T}AP\in\mathbb{R}^{n_{\rm c}\times n_{\rm c}}$ is known as the \textit{Galerkin coarse-level matrix}.

Two-level methods can be obtained by combining~\eqref{CR} and~\eqref{correction}. A symmetric two-level scheme (i.e., the presmoothing and postsmoothing processes are performed in a symmetric way) for solving~\eqref{system} can be described by Algorithm~\ref{alg:TL}. It is called an \textit{exact} two-level method, because the Galerkin coarse-level system $A_{\rm c}\mathbf{e}_{\rm c}=\mathbf{r}_{\rm c}$ is solved exactly. Algorithm~\ref{alg:TL} covers several well-known numerical algorithms, such as two-grid methods~\cite{Falgout2005,Vassilevski2008}, two-level hierarchical basis methods~\cite{Yserentant1986,Bank1988,Vassilevski2008}, and alternating projection methods~\cite{vonNeumann1950,XZ2002}. One may easily check that a sufficient and necessary condition for Algorithm~\ref{alg:TL} to be convergent in $A$-norm is
\begin{equation}\label{rank}
\rank(S \,\ P)=n,
\end{equation}
that is, for any $\mathbf{v}\in\mathbb{R}^{n}$, there exist $\mathbf{v}_{\rm s}\in\mathbb{R}^{n_{\rm s}}$ and $\mathbf{v}_{\rm c}\in\mathbb{R}^{n_{\rm c}}$ such that $\mathbf{v}=S\mathbf{v}_{\rm s}+P\mathbf{v}_{\rm c}$. Such a decomposition is not necessarily a direct sum. Under the condition~\eqref{rank}, an identity for the convergence factor of Algorithm~\ref{alg:TL} was presented in~\cite[Theorem~4.1]{Falgout2005}; see~\cite{XZ2002,XZ2017} for an abstract version.

\begin{algorithm}[!htbp]

\caption{\ \textbf{Exact two-level method}}\label{alg:TL}

\smallskip

\begin{algorithmic}[1]

\State Presmoothing: $\mathbf{u}^{(1)}\gets\mathbf{u}^{(0)}+SM_{\rm s}^{-1}S^{T}\big(\mathbf{f}-A\mathbf{u}^{(0)}\big)$

\smallskip

\State Restriction: $\mathbf{r}_{\rm c}\gets P^{T}\big(\mathbf{f}-A\mathbf{u}^{(1)}\big)$

\smallskip

\State Coarse-level correction: $\mathbf{e}_{\rm c}\gets A_{\rm c}^{-1}\mathbf{r}_{\rm c}$

\smallskip

\State Prolongation: $\mathbf{u}^{(2)}\gets\mathbf{u}^{(1)}+P\mathbf{e}_{\rm c}$

\smallskip

\State Postsmoothing: $\mathbf{u}_{\rm TL}\gets\mathbf{u}^{(2)}+SM_{\rm s}^{-T}S^{T}\big(\mathbf{f}-A\mathbf{u}^{(2)}\big)$

\end{algorithmic}

\end{algorithm}

In practice, however, it is often too costly to solve the coarse-level system exactly, especially when its size is relatively large. Instead, without essential loss of convergence speed, one may find an approximate solution to the linear system $A_{\rm c}\mathbf{e}_{\rm c}=\mathbf{r}_{\rm c}$. A typical strategy is to apply Algorithm~\ref{alg:TL} recursively in the correction steps. More generally, one can replace $A_{\rm c}$ by an SPD matrix $B_{\rm c}\in\mathbb{R}^{n_{\rm c}\times n_{\rm c}}$. The resulting inexact two-grid methods have been analyzed in~\cite{Notay2007,XXF2021,XXF2022}. For real-world problems, some nonlinear methods, such as the conjugate gradient~\cite{Hestenes1952} and generalized minimal residual methods~\cite{Saad1986}, have been applied to solve the coarsest-grid systems in multigrid algorithms (see, e.g.,~\cite{Babich2010,Brannick2016,Brower2018,Whyte2020}), even though the corresponding theoretical analyses are scarce. Indeed, in the correction step, one should focus on the output of a coarse solver (namely, an approximation to $A_{\rm c}^{-1}\mathbf{r}_{\rm c}$) instead of the solver itself. Usually, it is difficult to know the explicit form of a nonlinear coarse solver, which approximates $A_{\rm c}^{-1}$ in an \textit{implicit} way. As a result, some key quantities, like the extreme eigenvalues of $B_{\rm c}^{-1}A_{\rm c}$, involved in the aforementioned inexact two-grid theory will be unavailable.

In this paper, we are concerned with the convergence of two-level methods with general coarse solvers. The coarse-level system $A_{\rm c}\mathbf{e}_{\rm c}=\mathbf{r}_{\rm c}$ is solved approximately by an inner iterative procedure, in which linear, nonlinear, deterministic, randomized, or hybrid solvers can be used, as long as the corresponding accuracy estimates are available. Under a weak assumption on the accuracy of coarse solvers, we establish a general convergence theory for inexact two-level methods. The convergence theory extends the existing ones and can be used to guide the design of practical multilevel algorithms. Furthermore, it enables us to develop new algorithms combined with randomization techniques.

The rest of this paper is organized as follows. In Section~\ref{sec:pre}, we propose an inexact two-level algorithm and review the convergence theory of Algorithm~\ref{alg:TL}. In Section~\ref{sec:conv}, we establish a general framework for analyzing the convergence of inexact two-level methods. In Section~\ref{sec:ex}, we provide two representative examples of coarse solver, followed by discussions on how they fit into the proposed framework. In Section~\ref{sec:con}, we give some concluding remarks.

\section{Preliminaries} \label{sec:pre}

In this section, we introduce an inexact two-level algorithm for solving~\eqref{system} and an identity for the convergence factor of Algorithm~\ref{alg:TL}. For convenience, we first list some notation used in the subsequent discussions.

\begin{itemize}

\item[--] $I_{n}$ denotes the $n\times n$ identity matrix (or $I$ when its size is clear from context).

\item[--] $\lambda_{\min}(\cdot)$, $\lambda_{\min}^{+}(\cdot)$, and $\lambda_{\max}(\cdot)$ stand for the smallest eigenvalue, the smallest positive eigenvalue, and the largest eigenvalue of a matrix, respectively.

\item[--] $\lambda(\cdot)$ denotes the spectrum of a matrix.

\item[--] $\rho(\cdot)$ denotes the spectral radius of a matrix.

\item[--] $\|\cdot\|_{2}$ denotes the spectral norm of a matrix.

\item[--] $\|\cdot\|_{A}$ denotes the energy norm induced by an SPD matrix $A\in\mathbb{R}^{n\times n}$: for any $\mathbf{v}\in\mathbb{R}^{n}$, $\|\mathbf{v}\|_{A}=\sqrt{\mathbf{v}^{T}A\mathbf{v}}$; for any $B\in\mathbb{R}^{n\times n}$, $\|B\|_{A}=\max\limits_{\mathbf{v}\in\mathbb{R}^{n}\backslash\{0\}}\frac{\|B\mathbf{v}\|_{A}}{\|\mathbf{v}\|_{A}}$.

\item[--] $\tr(\cdot)$ denotes the trace of a matrix.

\item[--] $\mathbb{E}[\cdot]$ denotes the expectation of a random variable.

\end{itemize}

\subsection{Two-level methods}

Several fundamental assumptions involved in the analysis of two-level methods are summarized as follows.

\begin{itemize}

\item Let $S\in\mathbb{R}^{n\times n_{\rm s}}$ and $P\in\mathbb{R}^{n\times n_{\rm c}}$ be of full column rank, and let $(S \ P)$ be of full row rank, where
    \begin{displaymath}
    \max\{n_{\rm s},n_{\rm c}\}<n\leq n_{\rm s}+n_{\rm c}.
    \end{displaymath}

\item Let $M_{\rm s}\in\mathbb{R}^{n_{\rm s}\times n_{\rm s}}$ be a nonsingular matrix such that $M_{\rm s}+M_{\rm s}^{T}-A_{\rm s}$ is SPD, or, equivalently, $\|I-M_{\rm s}^{-1}A_{\rm s}\|_{A_{\rm s}}<1$, where $A_{\rm s}=S^{T}AS$.

\item For each $k=1,2,\ldots,\nu$, let $\mathscr{B}^{(k)}_{\rm c}\llbracket\cdot\rrbracket:\mathbb{R}^{n_{\rm c}}\rightarrow\mathbb{R}^{n_{\rm c}}$ be a general mapping that is expected to be a good approximation to $A_{\rm c}^{-1}$, where $A_{\rm c}=P^{T}AP$.

\end{itemize}

With the above assumptions, an inexact two-level method for solving~\eqref{system} can be described by Algorithm~\ref{alg:iTL}. The coarse-level correction step is an inner iterative procedure, which will be carried out $\nu$ iterations.

\begin{algorithm}[!htbp]

\caption{\ \textbf{Inexact two-level method}}\label{alg:iTL}

\smallskip

\begin{algorithmic}[1]

\State Presmoothing: $\mathbf{u}^{(1)}\gets\mathbf{u}^{(0)}+SM_{\rm s}^{-1}S^{T}\big(\mathbf{f}-A\mathbf{u}^{(0)}\big)$

\smallskip

\State Restriction: $\mathbf{r}_{\rm c}\gets P^{T}\big(\mathbf{f}-A\mathbf{u}^{(1)}\big)$

\smallskip

\State Coarse-level correction: $\mathbf{e}_{\rm c}^{(k)}\gets\mathbf{e}_{\rm c}^{(k-1)}+\mathscr{B}_{\rm c}^{(k)}\big\llbracket\mathbf{r}_{\rm c}-A_{\rm c}\mathbf{e}_{\rm c}^{(k-1)}\big\rrbracket$ with $\mathbf{e}_{\rm c}^{(0)}=\mathbf{0}$

\smallskip

\State Prolongation: $\mathbf{u}^{(2)}\gets\mathbf{u}^{(1)}+P\mathbf{e}_{\rm c}^{(\nu)}$

\smallskip

\State Postsmoothing: $\mathbf{u}_{\rm ITL}\gets\mathbf{u}^{(2)}+SM_{\rm s}^{-T}S^{T}\big(\mathbf{f}-A\mathbf{u}^{(2)}\big)$

\end{algorithmic}

\end{algorithm}

\begin{remark}
In the third step of Algorithm~\ref{alg:iTL}, the solvers $\big\{\mathscr{B}^{(k)}_{\rm c}\llbracket\cdot\rrbracket\big\}_{k=1}^{\nu}$ can be chosen flexibly and independently. This enables us to design some hybrid algorithms, which are expected to combine the advantages of different methods.
\end{remark}

The third step of Algorithm~\ref{alg:iTL} aims to find an approximate solution to the coarse-level system $A_{\rm c}\mathbf{e}_{\rm c}=\mathbf{r}_{\rm c}$. In particular, if $\mathbf{e}_{\rm c}^{(\nu)}=A_{\rm c}^{-1}\mathbf{r}_{\rm c}$, then Algorithm~\ref{alg:iTL} reduces to Algorithm~\ref{alg:TL}. In this case, for notational consistency, the output of Algorithm~\ref{alg:iTL} is denoted by $\mathbf{u}_{\rm TL}$, and
\begin{displaymath}
\mathbf{u}-\mathbf{u}_{\rm TL}=E_{\rm TL}\big(\mathbf{u}-\mathbf{u}^{(0)}\big),
\end{displaymath}
where $E_{\rm TL}$, called the \textit{iteration matrix} (or \textit{error propagation matrix}) of Algorithm~\ref{alg:TL}, is given by
\begin{equation}\label{E-TL1}
E_{\rm TL}=(I-SM_{\rm s}^{-T}S^{T}A)(I-\varPi_{A})(I-SM_{\rm s}^{-1}S^{T}A)
\end{equation}
with
\begin{equation}\label{piA}
\varPi_{A}:=PA_{\rm c}^{-1}P^{T}A.
\end{equation}
The iteration matrix $E_{\rm TL}$ can be expressed as
\begin{equation}\label{E-TL2}
E_{\rm TL}=I-B_{\rm TL}^{-1}A,
\end{equation}
where
\begin{displaymath}
B_{\rm TL}^{-1}=S\overline{M}_{\rm s}^{-1}S^{T}+(I-SM_{\rm s}^{-T}S^{T}A)PA_{\rm c}^{-1}P^{T}(I-ASM_{\rm s}^{-1}S^{T})
\end{displaymath}
with
\begin{equation}\label{barMs}
\overline{M}_{\rm s}:=M_{\rm s}(M_{\rm s}+M_{\rm s}^{T}-A_{\rm s})^{-1}M_{\rm s}^{T}.
\end{equation}

\begin{remark}
It is easy to check that $B_{\rm TL}^{-1}$ admits the hierarchical expression
\begin{displaymath}
B_{\rm TL}^{-1}=\big(S \ \ P\big)\widehat{B}_{\rm TL}^{-1}\big(S \ \ P\big)^{T}
\end{displaymath}
with
\begin{displaymath}
\widehat{B}_{\rm TL}=\begin{pmatrix}
I_{n_{\rm s}} & 0 \\
P^{T}ASM_{\rm s}^{-1} & I_{n_{\rm c}}
\end{pmatrix}\begin{pmatrix}
\overline{M}_{\rm s} & 0 \\
0 & A_{\rm c}
\end{pmatrix}\begin{pmatrix}
I_{n_{\rm s}} & M_{\rm s}^{-T}S^{T}AP \\
0 & I_{n_{\rm c}}
\end{pmatrix}.
\end{displaymath}
Since $\widehat{B}_{\rm TL}$ is SPD and $\rank(S \,\ P)=n$, it follows that $B_{\rm TL}$ is SPD, which validates the notation $B_{\rm TL}^{-1}$ appearing in~\eqref{E-TL2}. In addition, we get from~\eqref{E-TL1} that $A^{\frac{1}{2}}E_{\rm TL}A^{-\frac{1}{2}}$ is symmetric positive semidefinite (SPSD), which, combined with~\eqref{E-TL2}, yields the positive semidefiniteness of $B_{\rm TL}-A$.
\end{remark}

According to the positive semidefiniteness of $A^{\frac{1}{2}}E_{\rm TL}A^{-\frac{1}{2}}$ and~\eqref{E-TL2}, we deduce that
\begin{equation}\label{norm-ETL}
\|E_{\rm TL}\|_{A}=\rho(E_{\rm TL})=\lambda_{\max}(E_{\rm TL})=1-\lambda_{\min}\big(B_{\rm TL}^{-1}A\big),
\end{equation}
which is referred to as the \textit{convergence factor} of Algorithm~\ref{alg:TL}.

The following theorem gives an identity for characterizing the convergence factor $\|E_{\rm TL}\|_{A}$~\cite[Theorem~4.1]{Falgout2005}, an abstract version of which can be found in~\cite{XZ2002,XZ2017}.

\begin{theorem}
Let $\varPi_{A}$ be defined by~\eqref{piA}, and define
\begin{equation}\label{tildMs}
\widetilde{M}_{\rm s}:=M_{\rm s}^{T}(M_{\rm s}+M_{\rm s}^{T}-A_{\rm s})^{-1}M_{\rm s}.
\end{equation}
Then, the convergence factor of Algorithm~{\rm\ref{alg:TL}} can be characterized as
\begin{equation}\label{XZ-1}
\|E_{\rm TL}\|_{A}=1-\frac{1}{K_{\rm TL}},
\end{equation}
where
\begin{equation}\label{K-TL}
K_{\rm TL}=\sup_{\mathbf{v}\in\Range(I-\varPi_{A})}\inf_{\mathbf{v}_{\rm s}:\mathbf{v}=(I-\varPi_{A})S\mathbf{v}_{\rm s}}\frac{\mathbf{v}_{\rm s}^{T}\widetilde{M}_{\rm s}\mathbf{v}_{\rm s}}{\mathbf{v}^{T}A\mathbf{v}}.
\end{equation}
\end{theorem}

\begin{remark}
Since
\begin{displaymath}
\lambda_{\min}(E_{\rm TL})=\lambda_{\min}\big(A^{\frac{1}{2}}E_{\rm TL}A^{-\frac{1}{2}}\big)=0,
\end{displaymath}
we obtain
\begin{displaymath}
\lambda_{\max}\big(B_{\rm TL}^{-1}A\big)=1-\lambda_{\min}(E_{\rm TL})=1.
\end{displaymath}
From~\eqref{norm-ETL} and~\eqref{XZ-1}, we have
\begin{displaymath}
\lambda_{\min}\big(B_{\rm TL}^{-1}A\big)=\frac{1}{K_{\rm TL}}.
\end{displaymath}
Hence,
\begin{displaymath}
K_{\rm TL}=\frac{\lambda_{\max}\big(B_{\rm TL}^{-1}A\big)}{\lambda_{\min}\big(B_{\rm TL}^{-1}A\big)},
\end{displaymath}
i.e., $K_{\rm TL}$ is the corresponding condition number when Algorithm~\ref{alg:TL} is treated as a preconditioning method.
\end{remark}

\subsection{Two-grid methods}

In this subsection, we consider an extreme case of Algorithm~\ref{alg:iTL}. Let $M\in\mathbb{R}^{n\times n}$ be a nonsingular matrix such that $M+M^{T}-A$ is SPD. If $n_{\rm s}=n$, $S=I_{n}$, and $M_{\rm s}=M$, then Algorithm~\ref{alg:iTL} reduces to the following \textit{inexact two-grid method}.

\begin{algorithm}[!htbp]

\caption{\ \textbf{Inexact two-grid method}}\label{alg:iTG}

\smallskip

\begin{algorithmic}[1]

\State Presmoothing: $\mathbf{u}^{(1)}\gets\mathbf{u}^{(0)}+M^{-1}\big(\mathbf{f}-A\mathbf{u}^{(0)}\big)$

\smallskip

\State Restriction: $\mathbf{r}_{\rm c}\gets P^{T}\big(\mathbf{f}-A\mathbf{u}^{(1)}\big)$

\smallskip

\State Coarse-grid correction: $\mathbf{e}_{\rm c}^{(k)}\gets\mathbf{e}_{\rm c}^{(k-1)}+\mathscr{B}_{\rm c}^{(k)}\big\llbracket\mathbf{r}_{\rm c}-A_{\rm c}\mathbf{e}_{\rm c}^{(k-1)}\big\rrbracket$ with $\mathbf{e}_{\rm c}^{(0)}=\mathbf{0}$

\smallskip

\State Prolongation: $\mathbf{u}^{(2)}\gets\mathbf{u}^{(1)}+P\mathbf{e}_{\rm c}^{(\nu)}$

\smallskip

\State Postsmoothing: $\mathbf{u}_{\rm ITG}\gets\mathbf{u}^{(2)}+M^{-T}\big(\mathbf{f}-A\mathbf{u}^{(2)}\big)$

\end{algorithmic}

\end{algorithm}

If $\mathbf{e}_{\rm c}^{(\nu)}=A_{\rm c}^{-1}\mathbf{r}_{\rm c}$, then Algorithm~\ref{alg:iTG} is called an \textit{exact two-grid method}, in which case the iteration matrix takes the form
\begin{displaymath}
E_{\rm TG}=(I-M^{-T}A)(I-\varPi_{A})(I-M^{-1}A).
\end{displaymath}
Similarly, $E_{\rm TG}$ can be expressed as
\begin{displaymath}
E_{\rm TG}=I-B_{\rm TG}^{-1}A,
\end{displaymath}
where
\begin{displaymath}
B_{\rm TG}^{-1}=\overline{M}^{-1}+(I-M^{-T}A)PA_{\rm c}^{-1}P^{T}(I-AM^{-1})
\end{displaymath}
with
\begin{displaymath}
\overline{M}:=M(M+M^{T}-A)^{-1}M^{T}.
\end{displaymath}

Under the setting of two-grid methods, one can obtain a simplified version of the identity~\eqref{XZ-1}~\cite[Theorem~4.3]{Falgout2005}, as described in the following theorem.

\begin{theorem}
Define
\begin{displaymath}
\widetilde{M}:=M^{T}(M+M^{T}-A)^{-1}M \quad \text{and} \quad \varPi_{\widetilde{M}}:=P(P^{T}\widetilde{M}P)^{-1}P^{T}\widetilde{M}.
\end{displaymath}
Then, the convergence factor of the exact two-grid algorithm can be characterized as
\begin{equation}\label{XZ-2}
\|E_{\rm TG}\|_{A}=1-\frac{1}{K_{\rm TG}},
\end{equation}
where
\begin{equation}\label{K-TG}
K_{\rm TG}=\max_{\mathbf{v}\in\mathbb{R}^{n}\backslash\{0\}}\frac{\big\|\big(I-\varPi_{\widetilde{M}}\big)\mathbf{v}\big\|_{\widetilde{M}}^{2}}{\|\mathbf{v}\|_{A}^{2}}.
\end{equation}
\end{theorem}

\begin{remark}
Observe that $\varPi_{\widetilde{M}}$ is an $\widetilde{M}$-orthogonal projection onto the coarse space $\Range(P)$. For any $\mathbf{v}\in\mathbb{R}^{n}$, it holds that
\begin{displaymath}
\big\|\big(I-\varPi_{\widetilde{M}}\big)\mathbf{v}\big\|_{\widetilde{M}}=\min_{\mathbf{v}_{\rm c}\in\mathbb{R}^{n_{\rm c}}}\|\mathbf{v}-P\mathbf{v}_{\rm c}\|_{\widetilde{M}}.
\end{displaymath}
That is, measured by $\|\cdot\|_{\widetilde{M}}$, $\varPi_{\widetilde{M}}\mathbf{v}$ is the best choice for approximating $\mathbf{v}$ by a vector in $\Range(P)$. The identity~\eqref{XZ-2} is a powerful tool for analyzing two-grid methods (see, e.g.,~\cite{Falgout2005,XZ2017,Brannick2018,XXF2018}), which reflects the interplay between smoother and coarse space.
\end{remark}

\section{Convergence analysis} \label{sec:conv}

In this section, we present a convergence analysis of Algorithm~\ref{alg:iTL}: an upper bound for the error $\|\mathbf{u}-\mathbf{u}_{\rm ITL}\|_{A}$ is derived. In addition, we will analyze the convergence of a simplified two-level algorithm: Algorithm~\ref{alg:iTL} without postsmoothing.

\subsection{Convergence of Algorithm~\ref{alg:iTL}}

Recall that, in Algorithm~\ref{alg:iTL}, the coarse-level system to be solved reads
\begin{displaymath}
A_{\rm c}\mathbf{e}_{\rm c}=\mathbf{r}_{\rm c}.
\end{displaymath}
From the third step of Algorithm~\ref{alg:iTL}, we have
\begin{displaymath}
\mathbf{r}_{\rm c}-A_{\rm c}\mathbf{e}_{\rm c}^{(k)}=\mathbf{r}_{\rm c}-A_{\rm c}\mathbf{e}_{\rm c}^{(k-1)}-A_{\rm c}\mathscr{B}_{\rm c}^{(k)}\big\llbracket\mathbf{r}_{\rm c}-A_{\rm c}\mathbf{e}_{\rm c}^{(k-1)}\big\rrbracket.
\end{displaymath}
Let
\begin{equation}\label{phi}
\Phi_{\rm c}^{(k)}\llbracket\cdot\rrbracket=I_{n_{\rm c}}(\cdot)-A_{\rm c}\mathscr{B}_{\rm c}^{(k)}\llbracket\cdot\rrbracket \quad \forall\,k=1,2,\ldots,\nu.
\end{equation}
Then
\begin{displaymath}
\mathbf{r}_{\rm c}-A_{\rm c}\mathbf{e}_{\rm c}^{(k)}=\Phi_{\rm c}^{(k)}\big\llbracket\mathbf{r}_{\rm c}-A_{\rm c}\mathbf{e}_{\rm c}^{(k-1)}\big\rrbracket,
\end{displaymath}
which leads to
\begin{displaymath}
\mathbf{r}_{\rm c}-A_{\rm c}\mathbf{e}_{\rm c}^{(k)}=\big(\Phi_{\rm c}^{(k)}\circ\cdots\circ\Phi_{\rm c}^{(1)}\big)\big\llbracket\mathbf{r}_{\rm c}-A_{\rm c}\mathbf{e}_{\rm c}^{(0)}\big\rrbracket=\big(\Phi_{\rm c}^{(k)}\circ\cdots\circ\Phi_{\rm c}^{(1)}\big)\llbracket\mathbf{r}_{\rm c}\rrbracket.
\end{displaymath}
Here, the symbol $\circ$ denotes the composition of functions. More specifically,
\begin{displaymath}
(\varphi\circ \psi)(x)=\varphi\big(\psi(x)\big)
\end{displaymath}
for all $x$ in the domain of $\psi$. Define
\begin{equation}\label{rk}
\mathbf{r}_{k}:=\begin{cases}
\mathbf{r}_{\rm c} & \text{if $k=0$},\\[2pt]
\big(\Phi_{\rm c}^{(k)}\circ\cdots\circ\Phi_{\rm c}^{(1)}\big)\llbracket\mathbf{r}_{\rm c}\rrbracket & \text{if $k=1,2,\ldots,\nu$}.
\end{cases}
\end{equation}
Then, it holds that
\begin{equation}\label{diff}
A_{\rm c}^{-1}\mathbf{r}_{\rm c}-\mathbf{e}_{\rm c}^{(k)}=A_{\rm c}^{-1}\mathbf{r}_{k} \quad \forall\,k=0,1,\ldots,\nu.
\end{equation}

In light of~\eqref{diff}, we can prove the following accuracy estimate.

\begin{lemma}\label{lem:accu}
Let $\{\mathbf{r}_{k}\}_{k=0}^{\nu}$ be defined by~\eqref{rk}. If
\begin{equation}\label{rela-error}
\big\|A_{\rm c}^{-1}\mathbf{r}_{k-1}-\mathscr{B}_{\rm c}^{(k)}\llbracket\mathbf{r}_{k-1}\rrbracket\big\|_{A_{\rm c}}\leq\varepsilon_{k}\|\mathbf{r}_{k-1}\|_{A_{\rm c}^{-1}}
\end{equation}
for some $\varepsilon_{k}\in[0,1)$ and all $k=1,2,\ldots,\nu$, then
\begin{equation}\label{accu-est}
\big\|A_{\rm c}^{-1}\mathbf{r}_{\rm c}-\mathbf{e}_{\rm c}^{(\nu)}\big\|_{A_{\rm c}}\leq\varepsilon\|\mathbf{r}_{\rm c}\|_{A_{\rm c}^{-1}},
\end{equation}
where
\begin{equation}\label{eps}
\varepsilon=\prod_{k=1}^{\nu}\varepsilon_{k}.
\end{equation}
\end{lemma}

\begin{proof}
By~\eqref{phi}, \eqref{rk}, and~\eqref{rela-error}, we have that, for any $k=2,\ldots,\nu$,
\begin{align*}
\|\mathbf{r}_{k}\|_{A_{\rm c}^{-1}}&=\big\|\big(\Phi_{\rm c}^{(k)}\circ\cdots\circ\Phi_{\rm c}^{(1)}\big)\llbracket\mathbf{r}_{\rm c}\rrbracket\big\|_{A_{\rm c}^{-1}}\\
&=\big\|\big(\Phi_{\rm c}^{(k-1)}\circ\cdots\circ\Phi_{\rm c}^{(1)}\big)\llbracket\mathbf{r}_{\rm c}\rrbracket-A_{\rm c}\mathscr{B}_{\rm c}^{(k)}\big\llbracket\big(\Phi_{\rm c}^{(k-1)}\circ\cdots\circ\Phi_{\rm c}^{(1)}\big)\llbracket\mathbf{r}_{\rm c}\rrbracket\big\rrbracket\big\|_{A_{\rm c}^{-1}}\\
&=\big\|\mathbf{r}_{k-1}-A_{\rm c}\mathscr{B}_{\rm c}^{(k)}\llbracket\mathbf{r}_{k-1}\rrbracket\big\|_{A_{\rm c}^{-1}}\\
&=\big\|A_{\rm c}^{-1}\mathbf{r}_{k-1}-\mathscr{B}_{\rm c}^{(k)}\llbracket\mathbf{r}_{k-1}\rrbracket\big\|_{A_{\rm c}}\\
&\leq\varepsilon_{k}\|\mathbf{r}_{k-1}\|_{A_{\rm c}^{-1}}.
\end{align*}
Note that the above inequality also holds for $k=1$, i.e.,
\begin{displaymath}
\|\mathbf{r}_{1}\|_{A_{\rm c}^{-1}}\leq\varepsilon_{1}\|\mathbf{r}_{0}\|_{A_{\rm c}^{-1}}=\varepsilon_{1}\|\mathbf{r}_{\rm c}\|_{A_{\rm c}^{-1}}.
\end{displaymath}
In fact,
\begin{displaymath}
\|\mathbf{r}_{1}\|_{A_{\rm c}^{-1}}=\big\|A_{\rm c}^{-1}\Phi_{\rm c}^{(1)}\llbracket\mathbf{r}_{\rm c}\rrbracket\big\|_{A_{\rm c}}=\big\|A_{\rm c}^{-1}\mathbf{r}_{\rm c}-\mathscr{B}_{\rm c}^{(1)}\llbracket\mathbf{r}_{\rm c}\rrbracket\big\|_{A_{\rm c}}\leq\varepsilon_{1}\|\mathbf{r}_{\rm c}\|_{A_{\rm c}^{-1}},
\end{displaymath}
where we have used the condition~\eqref{rela-error} for $k=1$. Thus,
\begin{equation}\label{rk-recur}
\|\mathbf{r}_{k}\|_{A_{\rm c}^{-1}}\leq\varepsilon_{k}\|\mathbf{r}_{k-1}\|_{A_{\rm c}^{-1}} \quad \forall\,k=1,2,\ldots,\nu.
\end{equation}
Using~\eqref{diff} and~\eqref{rk-recur}, we obtain
\begin{displaymath}
\big\|A_{\rm c}^{-1}\mathbf{r}_{\rm c}-\mathbf{e}_{\rm c}^{(\nu)}\big\|_{A_{\rm c}}=\|\mathbf{r}_{\nu}\|_{A_{\rm c}^{-1}}\leq\varepsilon_{\nu}\|\mathbf{r}_{\nu-1}\|_{A_{\rm c}^{-1}}\leq\cdots\leq\Bigg(\prod_{k=1}^{\nu}\varepsilon_{k}\Bigg)\|\mathbf{r}_{\rm c}\|_{A_{\rm c}^{-1}}.
\end{displaymath}
This completes the proof.
\end{proof}

\begin{remark}
The condition~\eqref{rela-error} can be expressed as
\begin{displaymath}
\frac{\big\|A_{\rm c}^{-1}\mathbf{r}_{k-1}-\mathscr{B}_{\rm c}^{(k)}\llbracket\mathbf{r}_{k-1}\rrbracket\big\|_{A_{\rm c}}}{\big\|A_{\rm c}^{-1}\mathbf{r}_{k-1}\big\|_{A_{\rm c}}}\leq\varepsilon_{k},
\end{displaymath}
which characterizes the relative accuracy of the solver $\mathscr{B}_{\rm c}^{(k)}\llbracket\cdot\rrbracket$.
\end{remark}

\begin{remark}
The relation~\eqref{accu-est} is equivalent to
\begin{displaymath}
\|\mathbf{r}_{\rm c}\|_{A_{\rm c}^{-1}}^{2}-2\mathbf{r}_{\rm c}^{T}\mathbf{e}_{\rm c}^{(\nu)}+\big\|\mathbf{e}_{\rm c}^{(\nu)}\big\|_{A_{\rm c}}^{2}\leq\varepsilon^{2}\|\mathbf{r}_{\rm c}\|_{A_{\rm c}^{-1}}^{2}.
\end{displaymath}
That is,
\begin{equation}\label{rTe-low}
\frac{1}{2}\big((1-\varepsilon^{2})\|\mathbf{r}_{\rm c}\|_{A_{\rm c}^{-1}}^{2}+\big\|\mathbf{e}_{\rm c}^{(\nu)}\big\|_{A_{\rm c}}^{2}\big)\leq\mathbf{r}_{\rm c}^{T}\mathbf{e}_{\rm c}^{(\nu)}.
\end{equation}
By the Cauchy--Schwarz inequality, we have
\begin{equation}\label{CS}
\mathbf{r}_{\rm c}^{T}\mathbf{e}_{\rm c}^{(\nu)}\leq\|\mathbf{r}_{\rm c}\|_{A_{\rm c}^{-1}}\big\|\mathbf{e}_{\rm c}^{(\nu)}\big\|_{A_{\rm c}}.
\end{equation}
According to~\eqref{rTe-low} and~\eqref{CS}, we deduce that
\begin{displaymath}
\big\|\mathbf{e}_{\rm c}^{(\nu)}\big\|_{A_{\rm c}}^{2}-2\|\mathbf{r}_{\rm c}\|_{A_{\rm c}^{-1}}\big\|\mathbf{e}_{\rm c}^{(\nu)}\big\|_{A_{\rm c}}+(1-\varepsilon^{2})\|\mathbf{r}_{\rm c}\|_{A_{\rm c}^{-1}}^{2}\leq 0,
\end{displaymath}
which yields
\begin{equation}\label{e-low-up}
(1-\varepsilon)\|\mathbf{r}_{\rm c}\|_{A_{\rm c}^{-1}}\leq\big\|\mathbf{e}_{\rm c}^{(\nu)}\big\|_{A_{\rm c}}\leq(1+\varepsilon)\|\mathbf{r}_{\rm c}\|_{A_{\rm c}^{-1}}.
\end{equation}
Combining~\eqref{rTe-low}, \eqref{CS}, and~\eqref{e-low-up}, we obtain
\begin{displaymath}
(1-\varepsilon)\|\mathbf{r}_{\rm c}\|_{A_{\rm c}^{-1}}^{2}\leq\mathbf{r}_{\rm c}^{T}\mathbf{e}_{\rm c}^{(\nu)}\leq(1+\varepsilon)\|\mathbf{r}_{\rm c}\|_{A_{\rm c}^{-1}}^{2}.
\end{displaymath}
Such a relation is much weaker than the spectral equivalence relation
\begin{displaymath}
(1-\varepsilon)\mathbf{v}_{\rm c}^{T}A_{\rm c}^{-1}\mathbf{v}_{\rm c}\leq\mathbf{v}_{\rm c}^{T}B_{\rm c}^{-1}\mathbf{v}_{\rm c}\leq(1+\varepsilon)\mathbf{v}_{\rm c}^{T}A_{\rm c}^{-1}\mathbf{v}_{\rm c} \quad \forall\,\mathbf{v}_{\rm c}\in\mathbb{R}^{n_{\rm c}},
\end{displaymath}
or, equivalently,
\begin{displaymath}
\lambda\big(B_{\rm c}^{-1}A_{\rm c}\big)\subset[1-\varepsilon,1+\varepsilon],
\end{displaymath}
where $B_{\rm c}\in\mathbb{R}^{n_{\rm c}\times n_{\rm c}}$ is SPD.
\end{remark}

The following lemma gives a technical eigenvalue identity used in the subsequent analysis.

\begin{lemma}\label{lem:XZ-c}
Let $\varPi_{A}$ and $\widetilde{M}_{\rm s}$ be defined by~\eqref{piA} and~\eqref{tildMs}, respectively. Then
\begin{equation}\label{XZ-c}
\lambda_{\max}\big((I-S\widetilde{M}_{\rm s}^{-1}S^{T}A)\varPi_{A}\big)=\begin{cases}
1-\mu_{\rm TL} & \text{if $\rank(S^{T}AP)=n_{\rm c}$},\\[2pt]
1 & \text{if $\rank(S^{T}AP)<n_{\rm c}$},
\end{cases}
\end{equation}
where
\begin{equation}\label{muTL}
\mu_{\rm TL}=\lambda_{\min}^{+}\big(S\widetilde{M}_{\rm s}^{-1}S^{T}A\varPi_{A}\big).
\end{equation}
\end{lemma}

\begin{proof}
From~\eqref{tildMs}, we have
\begin{displaymath}
I-A_{\rm s}^{\frac{1}{2}}\widetilde{M}_{\rm s}^{-1}A_{\rm s}^{\frac{1}{2}}=\big(I-A_{\rm s}^{\frac{1}{2}}M_{\rm s}^{-1}A_{\rm s}^{\frac{1}{2}}\big)\big(I-A_{\rm s}^{\frac{1}{2}}M_{\rm s}^{-T}A_{\rm s}^{\frac{1}{2}}\big),
\end{displaymath}
which implies that $A_{\rm s}^{-1}-\widetilde{M}_{\rm s}^{-1}$ is SPSD. It follows that $I-A^{\frac{1}{2}}S\widetilde{M}_{\rm s}^{-1}S^{T}A^{\frac{1}{2}}$ is SPSD, which yields the positive semidefiniteness of $A^{-1}-S\widetilde{M}_{\rm s}^{-1}S^{T}$. Then
\begin{displaymath}
\lambda\big((I-S\widetilde{M}_{\rm s}^{-1}S^{T}A)\varPi_{A}\big)=\lambda\big((A^{-1}-S\widetilde{M}_{\rm s}^{-1}S^{T})^{\frac{1}{2}}A\varPi_{A}(A^{-1}-S\widetilde{M}_{\rm s}^{-1}S^{T})^{\frac{1}{2}}\big)\subset[0,+\infty),
\end{displaymath}
where we have used the fact that $A\varPi_{A}$ is SPSD. Hence,
\begin{align*}
\lambda_{\max}\big((I-S\widetilde{M}_{\rm s}^{-1}S^{T}A)\varPi_{A}\big)&=\lambda_{\max}\big((I-S\widetilde{M}_{\rm s}^{-1}S^{T}A)PA_{\rm c}^{-1}P^{T}A\big)\\
&=\lambda_{\max}\big(P^{T}A(I-S\widetilde{M}_{\rm s}^{-1}S^{T}A)PA_{\rm c}^{-1}\big)\\
&=\lambda_{\max}\big(I-P^{T}AS\widetilde{M}_{\rm s}^{-1}S^{T}APA_{\rm c}^{-1}\big)\\
&=1-\lambda_{\min}\big(P^{T}AS\widetilde{M}_{\rm s}^{-1}S^{T}APA_{\rm c}^{-1}\big).
\end{align*}
Observe that
\begin{displaymath}
\lambda\big(P^{T}AS\widetilde{M}_{\rm s}^{-1}S^{T}APA_{\rm c}^{-1}\big)=\lambda\big(A_{\rm c}^{-\frac{1}{2}}P^{T}AS\widetilde{M}_{\rm s}^{-1}S^{T}APA_{\rm c}^{-\frac{1}{2}}\big)\subset[0,+\infty).
\end{displaymath}
If $\rank(S^{T}AP)=n_{\rm c}$, then $A_{\rm c}^{-\frac{1}{2}}P^{T}AS\widetilde{M}_{\rm s}^{-1}S^{T}APA_{\rm c}^{-\frac{1}{2}}$ is SPD, which implies that
\begin{displaymath}
0<\lambda_{\min}\big(P^{T}AS\widetilde{M}_{\rm s}^{-1}S^{T}APA_{\rm c}^{-1}\big)=\lambda_{\min}^{+}\big(S\widetilde{M}_{\rm s}^{-1}S^{T}A\varPi_{A}\big).
\end{displaymath}
If $\rank(S^{T}AP)<n_{\rm c}$, then $A_{\rm c}^{-\frac{1}{2}}P^{T}AS\widetilde{M}_{\rm s}^{-1}S^{T}APA_{\rm c}^{-\frac{1}{2}}$ is singular, which leads to
\begin{displaymath}
\lambda_{\min}\big(P^{T}AS\widetilde{M}_{\rm s}^{-1}S^{T}APA_{\rm c}^{-1}\big)=0.
\end{displaymath}
Thus, the identity~\eqref{XZ-c} is proved.
\end{proof}

Based on Lemmas~\ref{lem:accu} and~\ref{lem:XZ-c}, we can derive the following convergence estimate.

\begin{theorem}\label{thm:iTL1}
Under the condition~\eqref{rela-error}, the approximate solution, $\mathbf{u}_{\rm ITL}$, generated by Algorithm~{\rm\ref{alg:iTL}} satisfies that
\begin{equation}\label{iTL-est1}
\|\mathbf{u}-\mathbf{u}_{\rm ITL}\|_{A}\leq\sigma_{\rm ITL}\big\|\mathbf{u}-\mathbf{u}^{(0)}\big\|_{A}
\end{equation}
with
\begin{displaymath}
\sigma_{\rm ITL}=\begin{cases}
1-\frac{1}{K_{\rm TL}}+\varepsilon(1-\mu_{\rm TL}) & \text{if $\rank(S^{T}AP)=n_{\rm c}$},\\[2pt]
1-\frac{1}{K_{\rm TL}}+\varepsilon & \text{if $\rank(S^{T}AP)<n_{\rm c}$},
\end{cases}
\end{displaymath}
where $K_{\rm TL}$, $\varepsilon$, and $\mu_{\rm TL}$ are given by~\eqref{K-TL}, \eqref{eps}, and~\eqref{muTL}, respectively.
\end{theorem}

\begin{proof}
From the last two steps of Algorithm~\ref{alg:iTL}, we have
\begin{displaymath}
\mathbf{u}-\mathbf{u}_{\rm ITL}=\big(I-SM_{\rm s}^{-T}S^{T}A\big)\big(\mathbf{u}-\mathbf{u}^{(1)}-P\mathbf{e}_{\rm c}^{(\nu)}\big).
\end{displaymath}
In the case of exact coarse solver, one has
\begin{displaymath}
\mathbf{u}-\mathbf{u}_{\rm TL}=\big(I-SM_{\rm s}^{-T}S^{T}A\big)\big(\mathbf{u}-\mathbf{u}^{(1)}-PA_{\rm c}^{-1}\mathbf{r}_{\rm c}\big).
\end{displaymath}
Then
\begin{displaymath}
\mathbf{u}_{\rm TL}-\mathbf{u}_{\rm ITL}=\big(I-SM_{\rm s}^{-T}S^{T}A\big)P\big(A_{\rm c}^{-1}\mathbf{r}_{\rm c}-\mathbf{e}_{\rm c}^{(\nu)}\big).
\end{displaymath}
Using~\eqref{accu-est}, we obtain
\begin{align*}
\|\mathbf{u}_{\rm TL}-\mathbf{u}_{\rm ITL}\|_{A}&=\big\|A^{\frac{1}{2}}\big(I-SM_{\rm s}^{-T}S^{T}A\big)P\big(A_{\rm c}^{-1}\mathbf{r}_{\rm c}-\mathbf{e}_{\rm c}^{(\nu)}\big)\big\|_{2}\\
&\leq\big\|A^{\frac{1}{2}}\big(I-SM_{\rm s}^{-T}S^{T}A\big)PA_{\rm c}^{-\frac{1}{2}}\big\|_{2}\big\|A_{\rm c}^{-1}\mathbf{r}_{\rm c}-\mathbf{e}_{\rm c}^{(\nu)}\big\|_{A_{\rm c}}\\
&=\big\|A_{\rm c}^{-\frac{1}{2}}P^{T}\big(I-ASM_{\rm s}^{-1}S^{T}\big)A^{\frac{1}{2}}\big\|_{2}\big\|A_{\rm c}^{-1}\mathbf{r}_{\rm c}-\mathbf{e}_{\rm c}^{(\nu)}\big\|_{A_{\rm c}}\\
&\leq\varepsilon\big\|A_{\rm c}^{-\frac{1}{2}}P^{T}A^{\frac{1}{2}}\big(I-A^{\frac{1}{2}}SM_{\rm s}^{-1}S^{T}A^{\frac{1}{2}}\big)\big\|_{2}\|\mathbf{r}_{\rm c}\|_{A_{\rm c}^{-1}}.
\end{align*}
From the first two steps of Algorithm~\ref{alg:iTL}, we have
\begin{displaymath}
\mathbf{r}_{\rm c}=P^{T}A\big(I-SM_{\rm s}^{-1}S^{T}A\big)\big(\mathbf{u}-\mathbf{u}^{(0)}\big),
\end{displaymath}
and hence
\begin{align*}
\|\mathbf{r}_{\rm c}\|_{A_{\rm c}^{-1}}&=\big\|A_{\rm c}^{-\frac{1}{2}}P^{T}A\big(I-SM_{\rm s}^{-1}S^{T}A\big)\big(\mathbf{u}-\mathbf{u}^{(0)}\big)\big\|_{2}\\
&=\big\|A_{\rm c}^{-\frac{1}{2}}P^{T}A^{\frac{1}{2}}\big(I-A^{\frac{1}{2}}SM_{\rm s}^{-1}S^{T}A^{\frac{1}{2}}\big)A^{\frac{1}{2}}\big(\mathbf{u}-\mathbf{u}^{(0)}\big)\big\|_{2}\\
&\leq\big\|A_{\rm c}^{-\frac{1}{2}}P^{T}A^{\frac{1}{2}}\big(I-A^{\frac{1}{2}}SM_{\rm s}^{-1}S^{T}A^{\frac{1}{2}}\big)\big\|_{2}\big\|\mathbf{u}-\mathbf{u}^{(0)}\big\|_{A}.
\end{align*}
Note that
\begin{align*}
&\big\|A_{\rm c}^{-\frac{1}{2}}P^{T}A^{\frac{1}{2}}\big(I-A^{\frac{1}{2}}SM_{\rm s}^{-1}S^{T}A^{\frac{1}{2}}\big)\big\|_{2}^{2}\\
&=\lambda_{\max}\Big(\big(I-A^{\frac{1}{2}}SM_{\rm s}^{-T}S^{T}A^{\frac{1}{2}}\big)A^{\frac{1}{2}}PA_{\rm c}^{-1}P^{T}A^{\frac{1}{2}}\big(I-A^{\frac{1}{2}}SM_{\rm s}^{-1}S^{T}A^{\frac{1}{2}}\big)\Big)\\
&=\lambda_{\max}\Big(A^{-\frac{1}{2}}\big(I-A^{\frac{1}{2}}SM_{\rm s}^{-T}S^{T}A^{\frac{1}{2}}\big)A^{\frac{1}{2}}PA_{\rm c}^{-1}P^{T}A^{\frac{1}{2}}\big(I-A^{\frac{1}{2}}SM_{\rm s}^{-1}S^{T}A^{\frac{1}{2}}\big)A^{\frac{1}{2}}\Big)\\
&=\lambda_{\max}\big((I-SM_{\rm s}^{-T}S^{T}A)\varPi_{A}(I-SM_{\rm s}^{-1}S^{T}A)\big)\\
&=\lambda_{\max}\big((I-SM_{\rm s}^{-1}S^{T}A)(I-SM_{\rm s}^{-T}S^{T}A)\varPi_{A}\big)\\
&=\lambda_{\max}\big((I-S\widetilde{M}_{\rm s}^{-1}S^{T}A)\varPi_{A}\big).
\end{align*}
We then have
\begin{displaymath}
\|\mathbf{u}_{\rm TL}-\mathbf{u}_{\rm ITL}\|_{A}\leq\varepsilon\lambda_{\max}\big((I-S\widetilde{M}_{\rm s}^{-1}S^{T}A)\varPi_{A}\big)\big\|\mathbf{u}-\mathbf{u}^{(0)}\big\|_{A}.
\end{displaymath}
Thus,
\begin{align*}
\|\mathbf{u}-\mathbf{u}_{\rm ITL}\|_{A}&\leq\|\mathbf{u}-\mathbf{u}_{\rm TL}\|_{A}+\|\mathbf{u}_{\rm TL}-\mathbf{u}_{\rm ITL}\|_{A}\\
&=\big\|E_{\rm TL}\big(\mathbf{u}-\mathbf{u}^{(0)}\big)\big\|_{A}+\|\mathbf{u}_{\rm TL}-\mathbf{u}_{\rm ITL}\|_{A}\\
&\leq\big(\|E_{\rm TL}\|_{A}+\varepsilon\lambda_{\max}\big((I-S\widetilde{M}_{\rm s}^{-1}S^{T}A)\varPi_{A}\big)\big)\big\|\mathbf{u}-\mathbf{u}^{(0)}\big\|_{A}.
\end{align*}
The desired estimate then follows immediately from~\eqref{XZ-1} and~\eqref{XZ-c}.
\end{proof}

As a corollary of Theorem~\ref{thm:iTL1}, the following convergence estimate for Algorithm~\ref{alg:iTG} holds.

\begin{corollary}
Under the condition~\eqref{rela-error}, the approximate solution, $\mathbf{u}_{\rm ITG}$, generated by Algorithm~{\rm\ref{alg:iTG}} satisfies that
\begin{equation}\label{iTG-est1}
\|\mathbf{u}-\mathbf{u}_{\rm ITG}\|_{A}\leq\bigg(1-\frac{1}{K_{\rm TG}}+\varepsilon\big(1-\lambda_{\min}^{+}(\widetilde{M}^{-1}A\varPi_{A})\big)\bigg)\big\|\mathbf{u}-\mathbf{u}^{(0)}\big\|_{A},
\end{equation}
where $K_{\rm TG}$ and $\varepsilon$ are given by~\eqref{K-TG} and~\eqref{eps}, respectively.
\end{corollary}

\subsection{Convergence of Algorithm~\ref{alg:iTL} without postsmoothing}

In Algorithm~\ref{alg:TL}, the main purpose of postsmoothing is to preserve the symmetry of its iteration matrix in $A$-inner product, which brings a lot of convenience to the theoretical analysis of two-level methods. In practice, it is not necessary to perform the postsmoothing step when the presmoothing and coarse-level correction processes complement each other very well. In this subsection, we present a convergence analysis of Algorithm~\ref{alg:iTL} without postsmoothing, in which case the output is denoted by $\mathbf{u}^{(2)}$.

For Algorithm~\ref{alg:TL} without postsmoothing, its iteration matrix is of the form
\begin{displaymath}
(I-\varPi_{A})(I-SM_{\rm s}^{-1}S^{T}A),
\end{displaymath}
where $\varPi_{A}$ is defined by~\eqref{piA}. Accordingly, its convergence factor is
\begin{align*}
&\big\|(I-\varPi_{A})(I-SM_{\rm s}^{-1}S^{T}A)\big\|_{A}\\
&=\big\|\big(I-A^{\frac{1}{2}}\varPi_{A}A^{-\frac{1}{2}}\big)\big(I-A^{\frac{1}{2}}SM_{\rm s}^{-1}S^{T}A^{\frac{1}{2}}\big)\big\|_{2}\\
&=\lambda_{\max}^{\frac{1}{2}}\Big(\big(I-A^{\frac{1}{2}}SM_{\rm s}^{-T}S^{T}A^{\frac{1}{2}}\big)\big(I-A^{\frac{1}{2}}\varPi_{A}A^{-\frac{1}{2}}\big)\big(I-A^{\frac{1}{2}}SM_{\rm s}^{-1}S^{T}A^{\frac{1}{2}}\big)\Big)\\
&=\lambda_{\max}^{\frac{1}{2}}\Big(A^{-\frac{1}{2}}\big(I-A^{\frac{1}{2}}SM_{\rm s}^{-T}S^{T}A^{\frac{1}{2}}\big)\big(I-A^{\frac{1}{2}}\varPi_{A}A^{-\frac{1}{2}}\big)\big(I-A^{\frac{1}{2}}SM_{\rm s}^{-1}S^{T}A^{\frac{1}{2}}\big)A^{\frac{1}{2}}\Big)\\
&=\lambda_{\max}^{\frac{1}{2}}\big((I-SM_{\rm s}^{-T}S^{T}A)(I-\varPi_{A})(I-SM_{\rm s}^{-1}S^{T}A)\big)\\
&=\lambda_{\max}^{\frac{1}{2}}(E_{\rm TL})\\
&=\bigg(1-\frac{1}{K_{\rm TL}}\bigg)^{\frac{1}{2}},
\end{align*}
where, in the last equality, we have used the facts~\eqref{norm-ETL} and~\eqref{XZ-1}.

The following theorem provides a convergence estimate for Algorithm~\ref{alg:iTL} without postsmoothing.

\begin{theorem}\label{thm:iTL2}
Under the condition~\eqref{rela-error}, the approximate solution, $\mathbf{u}^{(2)}$, generated by Algorithm~{\rm\ref{alg:iTL}} without postsmoothing satisfies that
\begin{equation}\label{iTL-est2}
\big\|\mathbf{u}-\mathbf{u}^{(2)}\big\|_{A}\leq\bigg(1-\frac{1-\varepsilon^{2}}{K_{\rm TL}}\bigg)^{\frac{1}{2}}\big\|\mathbf{u}-\mathbf{u}^{(0)}\big\|_{A},
\end{equation}
where $K_{\rm TL}$ and $\varepsilon$ are given by~\eqref{K-TL} and~\eqref{eps}, respectively.
\end{theorem}

\begin{proof}
From the fourth step of Algorithm~\ref{alg:iTL}, we have
\begin{displaymath}
\mathbf{u}-\mathbf{u}^{(2)}=\mathbf{u}-\mathbf{u}^{(1)}-P\mathbf{e}_{\rm c}^{(\nu)}.
\end{displaymath}
Then
\begin{align*}
\big\|\mathbf{u}-\mathbf{u}^{(2)}\big\|_{A}^{2}&=\big(\mathbf{u}-\mathbf{u}^{(1)}-P\mathbf{e}_{\rm c}^{(\nu)}\big)^{T}A\big(\mathbf{u}-\mathbf{u}^{(1)}-P\mathbf{e}_{\rm c}^{(\nu)}\big)\\
&=\big(\mathbf{u}-\mathbf{u}^{(1)}\big)^{T}A\big(\mathbf{u}-\mathbf{u}^{(1)}\big)-2\big(\mathbf{u}-\mathbf{u}^{(1)}\big)^{T}AP\mathbf{e}_{\rm c}^{(\nu)}+\big\|\mathbf{e}_{\rm c}^{(\nu)}\big\|_{A_{\rm c}}^{2}\\
&=\big(\mathbf{u}-\mathbf{u}^{(1)}\big)^{T}A\big(\mathbf{u}-\mathbf{u}^{(1)}\big)-2\mathbf{r}_{\rm c}^{T}\mathbf{e}_{\rm c}^{(\nu)}+\big\|\mathbf{e}_{\rm c}^{(\nu)}\big\|_{A_{\rm c}}^{2}\\
&\leq\big(\mathbf{u}-\mathbf{u}^{(1)}\big)^{T}A\big(\mathbf{u}-\mathbf{u}^{(1)}\big)-(1-\varepsilon^{2})\|\mathbf{r}_{\rm c}\|_{A_{\rm c}^{-1}}^{2}\\
&=\big(\mathbf{u}-\mathbf{u}^{(1)}\big)^{T}A\big(\mathbf{u}-\mathbf{u}^{(1)}\big)-(1-\varepsilon^{2})\big(\mathbf{u}-\mathbf{u}^{(1)}\big)^{T}A\varPi_{A}\big(\mathbf{u}-\mathbf{u}^{(1)}\big)\\
&=\big(\mathbf{u}-\mathbf{u}^{(1)}\big)^{T}A\big(I-(1-\varepsilon^{2})\varPi_{A}\big)\big(\mathbf{u}-\mathbf{u}^{(1)}\big),
\end{align*}
where we have used the relation~\eqref{rTe-low}. Let
\begin{displaymath}
\varPi=A^{\frac{1}{2}}\varPi_{A}A^{-\frac{1}{2}}=A^{\frac{1}{2}}PA_{\rm c}^{-1}P^{T}A^{\frac{1}{2}}
\end{displaymath}
and
\begin{displaymath}
E=\big(I-A^{\frac{1}{2}}SM_{\rm s}^{-T}S^{T}A^{\frac{1}{2}}\big)\big(I-(1-\varepsilon^{2})\varPi\big)\big(I-A^{\frac{1}{2}}SM_{\rm s}^{-1}S^{T}A^{\frac{1}{2}}\big).
\end{displaymath}
We then have
\begin{align*}
&\big(\mathbf{u}-\mathbf{u}^{(1)}\big)^{T}A\big(I-(1-\varepsilon^{2})\varPi_{A}\big)\big(\mathbf{u}-\mathbf{u}^{(1)}\big)\\
&=\big(\mathbf{u}-\mathbf{u}^{(0)}\big)^{T}\big(I-ASM_{\rm s}^{-T}S^{T}\big)A\big(I-(1-\varepsilon^{2})\varPi_{A}\big)\big(I-SM_{\rm s}^{-1}S^{T}A\big)\big(\mathbf{u}-\mathbf{u}^{(0)}\big)\\
&=\big(\mathbf{u}-\mathbf{u}^{(0)}\big)^{T}A^{\frac{1}{2}}EA^{\frac{1}{2}}\big(\mathbf{u}-\mathbf{u}^{(0)}\big).
\end{align*}
It follows that
\begin{displaymath}
\big\|\mathbf{u}-\mathbf{u}^{(2)}\big\|_{A}^{2}\leq\big(\mathbf{u}-\mathbf{u}^{(0)}\big)^{T}A^{\frac{1}{2}}EA^{\frac{1}{2}}\big(\mathbf{u}-\mathbf{u}^{(0)}\big),
\end{displaymath}
which, together with the symmetry of $E$, leads to
\begin{equation}\label{u-u2-up}
\big\|\mathbf{u}-\mathbf{u}^{(2)}\big\|_{A}^{2}\leq\lambda_{\max}(E)\big\|\mathbf{u}-\mathbf{u}^{(0)}\big\|_{A}^{2}.
\end{equation}

It is easy to see that
\begin{displaymath}
E=\big(I-A^{\frac{1}{2}}SM_{\rm s}^{-T}S^{T}A^{\frac{1}{2}}\big)\big(\varepsilon^{2}I+(1-\varepsilon^{2})(I-\varPi)\big)\big(I-A^{\frac{1}{2}}SM_{\rm s}^{-1}S^{T}A^{\frac{1}{2}}\big).
\end{displaymath}
By the Weyl's theorem (see, e.g.,~\cite[Theorem~4.3.1]{Horn2013}), we have
\begin{align*}
\lambda_{\max}(E)&\leq\varepsilon^{2}\lambda_{\max}\Big(\big(I-A^{\frac{1}{2}}SM_{\rm s}^{-T}S^{T}A^{\frac{1}{2}}\big)\big(I-A^{\frac{1}{2}}SM_{\rm s}^{-1}S^{T}A^{\frac{1}{2}}\big)\Big)\\
&\quad+(1-\varepsilon^{2})\lambda_{\max}\Big(\big(I-A^{\frac{1}{2}}SM_{\rm s}^{-T}S^{T}A^{\frac{1}{2}}\big)(I-\varPi)\big(I-A^{\frac{1}{2}}SM_{\rm s}^{-1}S^{T}A^{\frac{1}{2}}\big)\Big)\\
&=\varepsilon^{2}\lambda_{\max}\big(I-A^{\frac{1}{2}}S\overline{M}_{\rm s}^{-1}S^{T}A^{\frac{1}{2}}\big)+(1-\varepsilon^{2})\lambda_{\max}\big(A^{\frac{1}{2}}E_{\rm TL}A^{-\frac{1}{2}}\big),
\end{align*}
where $E_{\rm TL}$ and $\overline{M}_{\rm s}$ are given by~\eqref{E-TL1} and~\eqref{barMs}, respectively. Since
\begin{displaymath}
\lambda_{\max}\big(I-A^{\frac{1}{2}}S\overline{M}_{\rm s}^{-1}S^{T}A^{\frac{1}{2}}\big)=1-\lambda_{\min}\big(A^{\frac{1}{2}}S\overline{M}_{\rm s}^{-1}S^{T}A^{\frac{1}{2}}\big)=1
\end{displaymath}
and
\begin{displaymath}
\lambda_{\max}\big(A^{\frac{1}{2}}E_{\rm TL}A^{-\frac{1}{2}}\big)=\|E_{\rm TL}\|_{A}=1-\frac{1}{K_{\rm TL}},
\end{displaymath}
we obtain
\begin{displaymath}
\lambda_{\max}(E)\leq\varepsilon^{2}+(1-\varepsilon^{2})\bigg(1-\frac{1}{K_{\rm TL}}\bigg)=1-\frac{1-\varepsilon^{2}}{K_{\rm TL}},
\end{displaymath}
which, combined with~\eqref{u-u2-up}, yields the estimate~\eqref{iTL-est2}.
\end{proof}

Similarly, the following convergence estimate for Algorithm~{\rm\ref{alg:iTG}} without postsmoothing holds.

\begin{theorem}
Under the condition~\eqref{rela-error}, the approximate solution, $\mathbf{u}^{(2)}$, generated by Algorithm~{\rm\ref{alg:iTG}} without postsmoothing satisfies that
\begin{equation}\label{iTG-est2}
\big\|\mathbf{u}-\mathbf{u}^{(2)}\big\|_{A}\leq\bigg(1-\frac{1-\varepsilon^{2}}{K_{\rm TG}}-\varepsilon^{2}\lambda_{\min}(\widetilde{M}^{-1}A)\bigg)^{\frac{1}{2}}\big\|\mathbf{u}-\mathbf{u}^{(0)}\big\|_{A},
\end{equation}
where $K_{\rm TG}$ and $\varepsilon$ are given by~\eqref{K-TG} and~\eqref{eps}, respectively.
\end{theorem}

\begin{proof}
Observe that
\begin{align*}
&\lambda_{\max}\Big(\big(I-A^{\frac{1}{2}}M^{-T}A^{\frac{1}{2}}\big)\big(I-A^{\frac{1}{2}}M^{-1}A^{\frac{1}{2}}\big)\Big)\\
&=\lambda_{\max}\Big(\big(I-A^{\frac{1}{2}}M^{-1}A^{\frac{1}{2}}\big)\big(I-A^{\frac{1}{2}}M^{-T}A^{\frac{1}{2}}\big)\Big)\\
&=\lambda_{\max}\big(I-A^{\frac{1}{2}}\widetilde{M}^{-1}A^{\frac{1}{2}}\big)\\
&=1-\lambda_{\min}\big(A^{\frac{1}{2}}\widetilde{M}^{-1}A^{\frac{1}{2}}\big)\\
&=1-\lambda_{\min}(\widetilde{M}^{-1}A).
\end{align*}
The remainder of this proof is similar to that of Theorem~\ref{thm:iTL2}.
\end{proof}

\section{Examples} \label{sec:ex}

The theoretical framework developed in Section~\ref{sec:conv} is applicable for various coarse solvers, provided that the corresponding accuracy estimates are available. In this section, as examples, we introduce two numerical methods for solving the coarse-level system
\begin{equation}\label{coarse}
A_{\rm c}\mathbf{e}_{\rm c}=\mathbf{r}_{\rm c}.
\end{equation}
Discussions on how their accuracy estimates fit into our framework are also given.

\subsection{Conjugate gradient method}

The conjugate gradient (CG) method~\cite{Hestenes1952} is a well-known numerical algorithm for solving SPD problems. To find an approximate solution to the linear system~\eqref{coarse}, we perform $\ell$ iterations of the CG method with initial guess $\mathbf{e}_{\rm\s CG}^{(0)}\in\mathbb{R}^{n_{\rm c}}$. The resulting approximation is denoted by $\mathbf{e}_{\rm\s CG}^{(\ell)}$, which satisfies the following estimate (see, e.g.,~\cite[Theorem~38.5]{Trefethen1997}):
\begin{equation}\label{CG-est}
\big\|A_{\rm c}^{-1}\mathbf{r}_{\rm c}-\mathbf{e}_{\rm\s CG}^{(\ell)}\big\|_{A_{\rm c}}\leq 2\bigg(\frac{\sqrt{\kappa_{\rm c}}-1}{\sqrt{\kappa_{\rm c}}+1}\bigg)^{\ell}\big\|A_{\rm c}^{-1}\mathbf{r}_{\rm c}-\mathbf{e}_{\rm\s CG}^{(0)}\big\|_{A_{\rm c}},
\end{equation}
where $\kappa_{\rm c}=\frac{\lambda_{\max}(A_{\rm c})}{\lambda_{\min}(A_{\rm c})}$ is the spectral condition number of $A_{\rm c}$. If $\kappa_{\rm c}$ is very large, one may apply the CG method to an equivalent preconditioned system.

Under the setting of Algorithm~\ref{alg:iTL}, if $\nu=1$ and $\mathscr{B}_{\rm c}^{(1)}\llbracket\cdot\rrbracket$ is taken to be the CG solver, then the tolerance factor in~\eqref{accu-est} is
\begin{displaymath}
\varepsilon=2\bigg(\frac{\sqrt{\kappa_{\rm c}}-1}{\sqrt{\kappa_{\rm c}}+1}\bigg)^{\ell},
\end{displaymath}
which is less than $1$ if $\ell>\Big(\log_{2}\frac{\sqrt{\kappa_{\rm c}}+1}{\sqrt{\kappa_{\rm c}}-1}\Big)^{-1}$.

\subsection{Randomized coordinate descent method}

Besides those classical (deterministic) methods, one can use some randomized methods to solve~\eqref{coarse}.

Let $(A_{\rm c})_{i:}$ and $(\mathbf{r}_{\rm c})_{i}$ denote the $i$th row of $A_{\rm c}$ and the $i$th entry of $\mathbf{r}_{\rm c}$, respectively. The randomized coordinate descent (RCD) method~\cite{Leventhal2010,Nesterov2012,Gower2015} applied to~\eqref{coarse} can be described as follows:
\begin{equation}\label{RCD}
\mathbf{e}_{\rm\s RCD}^{(j+1)}=\mathop{\arg\min}_{\mathbf{x}_{\rm c}\in\mathbb{R}^{n_{\rm c}}}\big\|\mathbf{x}_{\rm c}-\mathbf{e}_{\rm\s RCD}^{(j)}\big\|_{A_{\rm c}}^{2} \quad \text{subject to} \quad (A_{\rm c})_{i:}\mathbf{x}_{\rm c}=(\mathbf{r}_{\rm c})_{i},
\end{equation}
where $i\in\{1,\ldots,n_{\rm c}\}$ is chosen randomly, with probability $p_{i}$. The solution to~\eqref{RCD} is given by
\begin{displaymath}
\mathbf{e}_{\rm\s RCD}^{(j+1)}=\mathbf{e}_{\rm\s RCD}^{(j)}+\frac{(\mathbf{r}_{\rm c})_{i}-(A_{\rm c})_{i:}\mathbf{e}_{\rm\s RCD}^{(j)}}{(A_{\rm c})_{ii}}\mathbf{d}_{i},
\end{displaymath}
where $(A_{\rm c})_{ii}$ and $\mathbf{d}_{i}$ denote the $(i,i)$-entry of $A_{\rm c}$ and the $i$th column of $I_{n_{\rm c}}$, respectively.

With the probability distribution $p_{i}=\frac{(A_{\rm c})_{ii}}{\tr(A_{\rm c})}$ ($i=1,\ldots,n_{\rm c}$), one can show that (see~\cite{Leventhal2010,Gower2015})
\begin{equation}\label{RCD-est1}
\mathbb{E}\Big[\big\|A_{\rm c}^{-1}\mathbf{r}_{\rm c}-\mathbf{e}_{\rm\s RCD}^{(\ell)}\big\|_{A_{\rm c}}^{2}\Big]\leq\bigg(1-\frac{\lambda_{\min}(A_{\rm c})}{\tr(A_{\rm c})}\bigg)^{\ell}\big\|A_{\rm c}^{-1}\mathbf{r}_{\rm c}-\mathbf{e}_{\rm\s RCD}^{(0)}\big\|_{A_{\rm c}}^{2},
\end{equation}
where $\mathbf{e}_{\rm\s RCD}^{(0)}\in\mathbb{R}^{n_{\rm c}}$ is an initial guess. By the Cauchy--Schwarz inequality for random variables, we have
\begin{displaymath}
\mathbb{E}\big[\big\|A_{\rm c}^{-1}\mathbf{r}_{\rm c}-\mathbf{e}_{\rm\s RCD}^{(\ell)}\big\|_{A_{\rm c}}\big]\leq\Big(\mathbb{E}\Big[\big\|A_{\rm c}^{-1}\mathbf{r}_{\rm c}-\mathbf{e}_{\rm\s RCD}^{(\ell)}\big\|_{A_{\rm c}}^{2}\Big]\Big)^{\frac{1}{2}},
\end{displaymath}
which, together with~\eqref{RCD-est1}, yields
\begin{equation}\label{RCD-est2}
\mathbb{E}\big[\big\|A_{\rm c}^{-1}\mathbf{r}_{\rm c}-\mathbf{e}_{\rm\s RCD}^{(\ell)}\big\|_{A_{\rm c}}\big]\leq\bigg(1-\frac{\lambda_{\min}(A_{\rm c})}{\tr(A_{\rm c})}\bigg)^{\frac{\ell}{2}}\big\|A_{\rm c}^{-1}\mathbf{r}_{\rm c}-\mathbf{e}_{\rm\s RCD}^{(0)}\big\|_{A_{\rm c}}.
\end{equation}

Under the setting of Algorithm~\ref{alg:iTL}, if $\nu=1$ and $\mathscr{B}_{\rm c}^{(1)}\llbracket\cdot\rrbracket$ is taken to be the RCD solver, then the expected tolerance factor in~\eqref{accu-est} is
\begin{displaymath}
\varepsilon=\bigg(1-\frac{\lambda_{\min}(A_{\rm c})}{\tr(A_{\rm c})}\bigg)^{\frac{\ell}{2}}.
\end{displaymath}
Similarly to the proofs of Theorems~\ref{thm:iTL1} and~\ref{thm:iTL2}, one can easily derive upper bounds for $\mathbb{E}[\|\mathbf{u}-\mathbf{u}_{\rm ITL}\|_{A}]$ and $\mathbb{E}\big[\big\|\mathbf{u}-\mathbf{u}^{(2)}\big\|_{A}\big]$ based on the estimates~\eqref{RCD-est1} and~\eqref{RCD-est2}.

\textit{An extension}: Let $\Omega$ be a random subset of $\{1,\ldots,n_{\rm c}\}$, and let $I_{:\Omega}\in\mathbb{R}^{n_{\rm c}\times|\Omega|}$ be a column concatenation of the columns of $I_{n_{\rm c}}$ indexed by $\Omega$. The randomized block-coordinate descent (RBCD) method~\cite{Richtarik2014} (also called the randomized Newton method~\cite{Gower2015,Qu2016}) applied to~\eqref{coarse} can be described by
\begin{equation}\label{RBCD}
\mathbf{e}_{\rm\s RBCD}^{(j+1)}=\mathop{\arg\min}_{\mathbf{x}_{\rm c}\in\mathbb{R}^{n_{\rm c}}}\big\|\mathbf{x}_{\rm c}-\mathbf{e}_{\rm\s RBCD}^{(j)}\big\|_{A_{\rm c}}^{2} \quad \text{subject to} \quad I_{:\Omega}^{T}A_{\rm c}\mathbf{x}_{\rm c}=I_{:\Omega}^{T}\mathbf{r}_{\rm c}.
\end{equation}
It was shown in~\cite{Gower2015,Qu2016} that
\begin{displaymath}
\mathbb{E}\Big[\big\|A_{\rm c}^{-1}\mathbf{r}_{\rm c}-\mathbf{e}_{\rm\s RBCD}^{(\ell)}\big\|_{A_{\rm c}}^{2}\Big]\leq\Big(1-\lambda_{\min}\big(\mathbb{E}\big[I_{:\Omega}\big(I_{:\Omega}^{T}A_{\rm c}I_{:\Omega}\big)^{-1}I_{:\Omega}^{T}A_{\rm c}\big]\big)\Big)^{\ell}\big\|A_{\rm c}^{-1}\mathbf{r}_{\rm c}-\mathbf{e}_{\rm\s RBCD}^{(0)}\big\|_{A_{\rm c}}^{2},
\end{displaymath}
where $\mathbf{e}_{\rm\s RBCD}^{(0)}\in\mathbb{R}^{n_{\rm c}}$ is an initial guess. Then
\begin{displaymath}
\mathbb{E}\big[\big\|A_{\rm c}^{-1}\mathbf{r}_{\rm c}-\mathbf{e}_{\rm\s RBCD}^{(\ell)}\big\|_{A_{\rm c}}\big]\leq\Big(1-\lambda_{\min}\big(\mathbb{E}\big[I_{:\Omega}\big(I_{:\Omega}^{T}A_{\rm c}I_{:\Omega}\big)^{-1}I_{:\Omega}^{T}A_{\rm c}\big]\big)\Big)^{\frac{\ell}{2}}\big\|A_{\rm c}^{-1}\mathbf{r}_{\rm c}-\mathbf{e}_{\rm\s RBCD}^{(0)}\big\|_{A_{\rm c}},
\end{displaymath}
in which case the expected tolerance factor in~\eqref{accu-est} is
\begin{displaymath}
\varepsilon=\Big(1-\lambda_{\min}\big(\mathbb{E}\big[I_{:\Omega}\big(I_{:\Omega}^{T}A_{\rm c}I_{:\Omega}\big)^{-1}I_{:\Omega}^{T}A_{\rm c}\big]\big)\Big)^{\frac{\ell}{2}}.
\end{displaymath}

\begin{remark}
In the third step of Algorithm~\ref{alg:iTL}, we are allowed to use different solvers at each iteration. As a result, some hybrid algorithms can be designed by combining different methods. For example, if $\nu=2$, one may choose $\mathscr{B}_{\rm c}^{(1)}\llbracket\cdot\rrbracket$ and $\mathscr{B}_{\rm c}^{(2)}\llbracket\cdot\rrbracket$ as the CG and RCD solvers, respectively. In view of Lemma~\ref{lem:accu}, the corresponding accuracy estimate follows immediately from~\eqref{CG-est} and~\eqref{RCD-est2}.
\end{remark}

\section{Conclusions} \label{sec:con}

In this work, we establish a general framework for analyzing the convergence of inexact two-level methods, in which the coarse-level system is solved approximately by an inner iterative procedure. The framework allows us to use linear, nonlinear, deterministic, randomized, or hybrid solvers in the inner iterations, as long as the corresponding accuracy estimates are available. Two examples of coarse solver are also provided, followed by discussions on how their accuracy estimates fit into our framework. Motivated by the proposed theory, we expect to develop new multilevel algorithms in the future, especially combined with randomization techniques.

\section*{Acknowledgment}

The author is grateful to Prof.~Chen-Song Zhang (LSEC, Academy of Mathematics and Systems Science, Chinese Academy of Sciences) for his helpful suggestions.

\bibliographystyle{amsplain}

\begin{thebibliography}{10}

\bibitem {Babich2010} R. Babich, J. Brannick, R. C. Brower, M. A. Clark, T. A. Manteuffel, S. F. McCormick, J. C. Osborn, and C.  Rebbi, \textit{Adaptive multigrid algorithm for the lattice Wilson--Dirac operator}, Phys. Rev. Lett. \textbf{105} (2010), 201602.

\bibitem {Bank1988} R. E. Bank, T. F. Dupont, and H. Yserentant, \textit{The hierarchical basis multigrid method}, Numer. Math. \textbf{52} (1988), 427--458.

\bibitem {Brandt2000} A. E. Brandt, \textit{General highly accurate algebraic coarsening}, Electron. Trans. Numer. Anal. \textbf{10} (2000), 1--20.

\bibitem {Brannick2018} J. Brannick, F. Cao, K. Kahl, R. D. Falgout, and X. Hu, \textit{Optimal interpolation and compatible relaxation in classical algebraic multigrid}, SIAM J. Sci. Comput. \textbf{40} (2018), A1473--A1493.

\bibitem {Brannick2016} J. Brannick, A. Frommer, K. Kahl, B. Leder, M. Rottmann, and A. Strebel, \textit{Multigrid preconditioning for the overlap operator in lattice QCD}, Numer. Math. \textbf{132} (2016), 463--490.

\bibitem {Briggs2000} W. L. Briggs, V. E. Henson, and S. F. McCormick, \textit{A Multigrid Tutorial}, second ed., SIAM, Philadelphia, PA, 2000.

\bibitem {Brower2018} R. C. Brower, E. Weinberg, M. A. Clark, and A. Strelchenko, \textit{Multigrid algorithm for staggered lattice fermions}, Phys. Rev. D \textbf{97} (2018), 114513.

\bibitem {Falgout2004} R. D. Falgout and P. S. Vassilevski, \textit{On generalizing the algebraic multigrid framework}, SIAM J. Numer. Anal. \textbf{42} (2004), 1669--1693.

\bibitem {Falgout2005} R. D. Falgout, P. S. Vassilevski, and L. T. Zikatanov, \textit{On two-grid convergence estimates}, Numer. Linear Algebra Appl. \textbf{12} (2005), 471--494.

\bibitem {Gower2015} R. M. Gower and P. Richt\'{a}rik, \textit{Randomized iterative methods for linear systems}, SIAM J. Matrix Anal. Appl. \textbf{36} (2015), 1660--1690.

\bibitem {Hackbusch1985} W. Hackbusch, \textit{Multi-Grid Methods and Applications}, Springer-Verlag, Berlin, Heidelberg, 1985.

\bibitem {Hestenes1952} M. R. Hestenes and E. Stiefel, \textit{Methods of conjugate gradients for solving linear systems}, J. Res. Nat. Bur. Stand. \textbf{49} (1952), 409--436.

\bibitem {Horn2013} R. A. Horn and C. R. Johnson, \textit{Matrix Analysis}, second ed., Cambridge University Press, Cambridge, 2013.

\bibitem {Leventhal2010} D. Leventhal and A. S. Lewis, \textit{Randomized methods for linear constraints: Convergence rates and conditioning}, Math. Oper. Res. \textbf{35} (2010), 641--654.

\bibitem {Nesterov2012} Y. Nesterov, \textit{Efficiency of coordinate descent methods on huge-scale optimization problems}, SIAM J. Optim. \textbf{22} (2012), 341--362.

\bibitem {Notay2007} Y. Notay, \textit{Convergence analysis of perturbed two-grid and multigrid methods}, SIAM J. Numer. Anal. \textbf{45} (2007), 1035--1044.

\bibitem {Qu2016} Z. Qu, P. Richt\'{a}rik, M. Tak\'{a}\v{c}, and O. Fercoq, \textit{SDNA: Stochastic dual Newton ascent for empirical risk minimization}, in Proceedings of the 33rd International Conference on Machine Learning, New York, \textbf{48} (2016), 1823--1832.

\bibitem {Richtarik2014} P. Richt\'{a}rik and M. Tak\'{a}\v{c}, \textit{Iteration complexity of randomized block-coordinate descent methods for minimizing a composite function}, Math. Program., Ser. A \textbf{144} (2014), 1--38.

\bibitem {Saad1986} Y. Saad and M. H. Schultz, \textit{GMRES: A generalized minimal residual algorithm for solving nonsymmetric linear systems}, SIAM J. Sci. Stat. Comput. \textbf{7} (1986), 856--869.

\bibitem {Trefethen1997} L. N. Trefethen and D. Bau III, \textit{Numerical Linear Algebra}, SIAM, Philadelphia, PA, 1997.

\bibitem {Trottenberg2001} U. Trottenberg, C. W. Oosterlee, and A. Sch\"{u}ller, \textit{Multigrid}, Academic Press, London, 2001.

\bibitem {Vassilevski2008} P. S. Vassilevski, \textit{Multilevel Block Factorization Preconditioners: Matrix-based Analysis and Algorithms for Solving Finite Element Equations}, Springer, New York, 2008.

\bibitem {vonNeumann1950} J. von Neumann, \textit{Functional Operators, Volume~\uppercase\expandafter{\romannumeral2}: The Geometry of Orthogonal Spaces}, Annals of mathematics studies, no.~22, Princeton University Press, 1950.

\bibitem {Whyte2020} T. Whyte, W. Wilcox, and R. B. Morgan, \textit{Deflated GMRES with multigrid for lattice QCD}, Phys. Lett. B \textbf{803} (2020), 135281.

\bibitem {XZ2002} J. Xu and L. T. Zikatanov, \textit{The method of alternating projections and the method of subspace corrections in Hilbert space}, J. Amer. Math. Soc. \textbf{15} (2002), 573--597.

\bibitem {XZ2017} J. Xu and L. T. Zikatanov, \textit{Algebraic multigrid methods}, Acta Numer. \textbf{26} (2017), 591--721.

\bibitem {XXF2018} X. Xu and C.-S. Zhang, \textit{On the ideal interpolation operator in algebraic multigrid methods}, SIAM J. Numer. Anal. \textbf{56} (2018), 1693--1710.

\bibitem {XXF2021} X. Xu and C.-S. Zhang, \textit{Convergence analysis of inexact two-grid methods: A theoretical framework}, SIAM J. Numer. Anal., to appear.

\bibitem {XXF2022} X. Xu and C.-S. Zhang, \textit{A new analytical framework for the convergence of inexact two-grid methods}, submitted.

\bibitem {Yserentant1986} H. Yserentant, \textit{On the multilevel splitting of finite element spaces}, Numer. Math. \textbf{49} (1986), 379--412.

\end{thebibliography}

\end{document}